\documentclass[a4paper,12pt]{article}
\usepackage[margin=1in]{geometry}
\usepackage{amsmath,amsthm,amssymb,amsfonts}
\usepackage{hyperref}

\newtheorem{theorem}{Theorem}
\newtheorem{lemma}{Lemma}
\newtheorem{remark}{Remark}

\title{Explicit entropy bounds for symmetric nearest-neighbor subshifts}

\author{Vuong Bui\thanks{\texttt{bui.vuong@yandex.ru}}}
\date{}

\begin{document}

\maketitle

\begin{abstract}
    We provide another approach to Friedland's result that the topological entropy $h$ of a symmetric nearest-neighbor subshift is computable. Instead of the previous algebraic technique, our approach is mostly combinatorial and involves only counts of locally admissible patterns $C_n$ of a cube $[1,n]^d$ in $\mathbb Z^d$. The main idea is a reflection-gluing construction: we flip admissible patterns and merge them along their boundaries. In addition to a short and elementary proof, another advantage is that our approach yields an explicit convergence rate in arbitrary dimensions, whereas obtaining such a rate is already complicated for $\mathbb Z^3$ in Friedland's approach. In particular, we show that for every $n\ge 1$,
    \[
        \frac{1}{n^d}(\log C_{n+1} - q_d(n)\log|\Sigma|) \le h \le \frac{1}{n^d} \log C_n,
    \]
    where $\Sigma$ is the alphabet and
    \[
        q_d(n)=(2^d-1)\sum_{k=0}^{d-1} \frac{\binom{d}{k}}{2^d-2^k}\, n^k.
    \]
\end{abstract}

\section{Introduction}
The topological entropy of subshifts of finite type (SFTs) is a standard object in
symbolic dynamics. For general background, especially in one-dimensional symbolic dynamics,
we refer to \cite{lind2021introduction}.
This notion is however well known to be uncomputable in general for higher dimensions. In fact, it is already undecidable whether a given SFT is empty, due to Berger's seminal work on the domino problem \cite{berger1966undecidability}. In particular, although an empty SFT must already fail to admit a locally admissible pattern on some finite square, there is no computable bound, in general, on the size of the first such square. This phenomenon highlights a fundamental distinction between the plane and the one-dimensional case: the presence of one extra dimension allows the simulation of much more complicated models of computation, including a Turing machine. As a consequence, the emptiness problem for two-dimensional SFTs is closely related to the halting problem.

The connection between computability theory and SFTs is in fact much deeper. In a breakthrough result of Hochman and Meyerovitch \cite{hochman2010characterization}, it is shown that the class of topological entropies coincides precisely with the class of nonnegative real numbers that are right-recursively enumerable, that is, those admitting a computable approximation from above.
While the construction of an SFT realizing a given right-recursively enumerable number is highly nontrivial, the converse direction, that the topological entropy of every SFT is right-recursively enumerable, is considerably simpler.
This latter fact  was already known much earlier, going back to the work of Friedland \cite{friedland1997entropy}, and was later refined with a cleaner proof in \cite{hochman2010characterization}. In fact, these works show that the topological entropy coincides with the combinatorial entropy, which is defined purely in terms of the growth rate of locally admissible patterns on finite cubes. The combinatorial entropy is often more convenient to work with, since we just need to generate local configurations without the need to decide whether they can be extended to the entire plane, a problem which is itself undecidable. Local configurations are actually more commonly studied in combinatorics on words.

To prove the computability of topological entropy in certain restricted classes of SFTs, it therefore suffices to provide an explicit sequence of lower bounds converging to the entropy. For instance, in \cite{hochman2010characterization} this is achieved for strongly irreducible SFTs, and in \cite{pavlov2015entropies} for SFTs satisfying a strong mixing property. In both cases, the proofs proceed by constructing larger configurations and gluing them together, a process justified by the imposed restrictions. However, no explicit rate of convergence is obtained since these conditions are still quite general.

Throughout this article, all subshifts are assumed to be subshifts of finite type. Thus, for brevity, we will simply write ``subshift'' instead of ``subshift of finite type.''

To obtain stronger results, one may study simpler cases, for example, symmetric nearest-neighbor subshifts. Recall that a subshift is nearest-neighbor if the forbidden patterns involve only adjacent pairs of points. 
A subshift is called symmetric if it is invariant under
reflection in each coordinate direction.
More precisely, in our case of symmetric nearest-neighbor subshifts, if we let $F_i\subset \Sigma^2$ denote the set of forbidden adjacent pairs in the direction of the $i$-th coordinate, then
\[
(a,b)\in F_i \iff (b,a)\in F_i
\]
for all \(a,b\in\Sigma\) and all \(i=1,\ldots,d\).
Examples include hard square models (the alphabet is $\{0,1\}$ without two adjacent $1$s) and proper colorings (the alphabet is the set of colors without two adjacent points having the same color). Note that in these examples the values in the forbidden patterns are the same across all the coordinates but it is not required in a general symmetric nearest-neighbor subshift.
In other words, we do not assume that the sets 
$F_i$ are the same for different coordinates $i$.

Friedland \cite{friedland1997entropy} showed that for symmetric nearest-neighbor subshifts, the three terms: (i) the topological entropy, (ii) the combinatorial entropy, and (iii) the growth rate of periodic patterns all coincide. 
Recall that while the topological entropy is expressed in terms of the number of globally admissible patterns, the combinatorial one involves only locally admissible patterns.
Note that the number of locally admissible patterns on a finite cube gives an explicit upper bound on the entropy, while counts of periodic patterns provide lower bounds.
Therefore, we have the computability of the entropy as a corollary. While this approach is structurally deep (and beautiful) with the relation between several notions of entropy, the proof relies on delicate algebraic arguments. The convergence rate was provided for $\mathbb Z^2$ but it is complicated to obtain one for higher dimensions.

In this article, we propose a combinatorial approach that yields a shorter proof of the computability for symmetric nearest-neighbor subshifts. Our proof does not use algebraic tools as in the case of Friedland's proof, but mostly some intuitive geometric manipulations together with elementary estimates. In particular, the key idea is that we flip the patterns and glue them by merging the boundaries.
One advantage of our approach is that it works directly with counts of locally admissible patterns. 
These quantities are directly enumerable and avoid the additional step of
constructing and rigorously estimating the spectral radii of transfer matrices
as in Friedland's approach.
However, the main advantage of our method is an explicit convergence rate in arbitrary dimensions.

Let $\Sigma$ denote the alphabet, and $C_n$ denote the number of locally admissible patterns on the cube $[1,n]^d$.
In other words, it is the number of elements in $\Sigma^{[1,n]^d}$ so that no two adjacent points in $[1,n]^d$ form a forbidden pattern. 

Our approach in the article yields the following main result.
\begin{theorem} \label{thm:main}
    For every $n\ge 1$, the topological entropy $h$ of a symmetric nearest-neighbor subshift satisfies
    \[
        \frac{1}{n^d}(\log C_{n+1} - q_d(n)\log|\Sigma|) \le h \le \frac{1}{n^d} \log C_n,
    \]
    where
    \[
        q_d(n)=(2^d-1)\sum_{k=0}^{d-1} \frac{\binom{d}{k}}{2^d-2^k} \, n^k.
    \]
\end{theorem}

We assume throughout the article that the subshift is nonempty. 
However, in case of an empty subshift, the bounds in Theorem \ref{thm:main} still work with the convention that $h=-\infty$ and $\log 0 = -\infty$. Indeed, as we will discuss later, $C_n$ for $n\ge 2$ are either all zeros or all positive. Therefore, if $C_2=0$, then $h=-\infty$ and the lower bound is also $-\infty$ for every $n\ge 1$. On the other hand, $h=-\infty$ means that the upper bound always works.

The main contribution of Theorem \ref{thm:main} is the lower bound, since it is already well known that
\begin{equation} 
    h \le \frac{1}{n^d} \log C_n.
\end{equation}
Remark \ref{rem:concat} sketches the simple concatenation used to prove this upper bound, with some comparisons to the approach used for the lower bound.

Note that Theorem \ref{thm:main} does not violate the undecidability of the emptiness of arbitrary subshifts of finite type, since in our case the symmetry allows us to extend every locally admissible pattern on $[1,n]^d$ for $n\ge 2$ to a periodic pattern and hence to fill up the whole plane. This will be proved in Lemma \ref{lem:extending} in Section \ref{sec:plane}.  In that lemma, we also have $C_n\le C_{n+1}$ for $n\ge 2$. Moreover, the lemma  reproves that the combinatorial entropy and the topological entropy are identical for this kind of subshifts. While Lemma \ref{lem:extending} is not essential for the proof of Theorem \ref{thm:main}, we believe it is of independent interest, as it encapsulates the main ideas of the article. The lemma could be useful elsewhere.

The computability follows from Theorem \ref{thm:main}. Indeed, for $n\ge 2$, we have $C_n\le C_{n+1}$, hence the gap between the upper bound and the lower bound is at most
\[
    \frac{q_d(n)\log|\Sigma|}{n^d}.
\]
This gap is $O(1/n)$ as $q_d(n)$ has degree $d-1$.
More precisely, the leading term of \(q_d(n)\) is
\[
(2^d-1)\binom{d}{d-1}\frac{n^{d-1}}{2^d-2^{d-1}}
=
d\left(2-2^{1-d}\right)n^{d-1}.
\]
Therefore, the gap between the upper and lower bounds is
\[
\frac{q_d(n)\log|\Sigma|}{n^d}
=
\frac{d\left(2-2^{1-d}\right)\log|\Sigma|}{n}
+
O(n^{-2}).
\]

We first present the proof in the case of $\mathbb Z^2$, where the key ideas are most transparent, in Section \ref{sec:plane}. We then extend the argument to arbitrary dimensions, which requires more notation and some estimates, in Section \ref{sec:higher}.

\section{On the plane}
\label{sec:plane}
In this section, we show how to obtain Theorem \ref{thm:main} for $\mathbb Z^2$, which is rewritten as follows.
\begin{theorem} \label{thm:plane}
    For every $n\ge 1$, the topological entropy $h$ of a symmetric nearest-neighbor subshift satisfies
    \[
        \left(\frac{C_{n+1}}{|\Sigma|^{3n+1}}\right)^{\frac{1}{n^2}} \le e^h \le (C_n)^{\frac{1}{n^2}}.
    \]
\end{theorem}
We do not write it in logarithmic form since we want to work directly with the
counts.
We recall that the topological entropy is identical to the combinatorial entropy, which is the logarithm of the growth rate of the number of locally admissible patterns \cite{hochman2010characterization}. In other words, we have
\[
    e^h = \lim_{n\to\infty} (C_n)^{1/n^2}.
\]
The fact that the two entropies are identical can later be recovered in Lemma \ref{lem:extending}, making our proof self-contained.

Our proof uses locally admissible patterns only, and from now on, we call them admissible patterns for brevity. Note that Lemma \ref{lem:extending} (later in this section) implies that local admissibility and global admissibility are the same for cubes $[1,n]^d$ with $n\ge 2$ in our case of symmetric nearest-neighbor subshifts.

For a given admissible pattern $P$ on $[1,n]^2$, we call the values at
\[
    [1,n]^2\setminus [1,n-1]^2 = (\{n\}\times[1,n]) \cup ([1,n]\times\{n\})
\]
the state of $P$. In other words, the state of $P$ is the values at the top-most row and right-most column of $P$, as depicted by
\begin{center}
    {
    \ttfamily
    \begin{tabular}{|c|c|}
    \hline
        a b c d & e \\ \hline
     & \begin{tabular}{@{}c@{}} f\\ g \\ h \\ i \end{tabular} \\
     \hline
    \end{tabular}\;.
     }
\end{center}
For such a state $s$ we denote by $C_n^{(s)}$ the number of patterns with state $s$.

Let $H(P),V(P)$ denote the horizontal and vertical flips, respectively, of a pattern $P$. By a horizontal flip, we mean a pattern also on $[1,n]^2$ but the $i$-th row of $H(P)$ is the $(n+1-i)$-th row of $P$. In other words, we swap the $i$-th row and the $(n+1-i)$-th row. The vertical flip is defined similarly, using columns instead of rows. For example, the vertical flip of the above pattern $P$ is
\begin{center}
    {
    \ttfamily
    \begin{tabular}{|c|c|}
    \hline
        e &  d b c a \\ \hline
     \begin{tabular}{@{}c@{}} f\\ g \\ h \\ i \end{tabular} & \\
     \hline
    \end{tabular}\;.
     }
\end{center}

By the symmetry of the subshift, $H(P)$ and $V(P)$ are admissible if $P$ is admissible. It follows that the pattern $H(V(P))=V(H(P))$ is also admissible.

Translate $H(P),V(P),H(V(P))$ to the squares 
\[
    [1,n]\times [n,2n-1],\quad [n,2n-1]\times [1,n],\quad [n,2n-1]\times [n,2n-1].
\]
In other words, we translate them so that $P$ and $H(P)$ intersect at the top-most row of $P$ (which is also the bottom-most of $H(P)$). Likewise, for the other pairs, they intersect either at a row, or a column, or at the corner \texttt{e} as depicted in the following resulting pattern of size $[1,2n-1]\times [1,2n-1]$:
\begin{center}
{
\ttfamily
\begin{tabular}{|c|c|c|}
\hline
$H(P)$ & \begin{tabular}{@{}c@{}} i\\h\\g\\f \end{tabular} & $H(V(P))$ \\
\hline
a b c d & e & d c b a \\
\hline
$P$ & \begin{tabular}{@{}c@{}} f\\g\\h\\i \end{tabular} & $V(P)$ \\
\hline
\end{tabular}\;.
}
\end{center}
Every adjacent pair of sites in the resulting pattern is contained in at
least one of the four components. Since each component is
admissible, no forbidden adjacent pair can occur.
In other words, the resulting pattern is admissible.

We provide the following lemma to make the proof self-contained and to analyze the convergence rate. 
This lemma is not needed if one invokes the known equality of topological and combinatorial entropy.
\begin{lemma} \label{lem:extending}
For a symmetric nearest-neighbor subshift, we can extend every locally admissible pattern $P$ of $[1,n]^2$ for $n\ge 2$ to a periodic
pattern, and hence extend it to the entire plane.
As a corollary, the combinatorial entropy and the topological entropy are identical for symmetric nearest-neighbor subshifts.
On the other hand, $P$ can also be extended to a pattern of $[1,n+1]^2$, which implies $C_n\le C_{n+1}$ for $n\ge 2$. 
The same applies for higher dimensions.
\end{lemma}
\begin{proof}
    Reduce the resulting pattern $Q$ (from the above process applied to $P$), which is of size $[1,2n-1]^2$, to size $[1,2n-2]^2$ (that is we remove the top-most row and right-most column). We obtain a periodic pattern $Q'$, since the left-most column (resp. the bottom-most row) of $Q$ coincides with the right-most column (resp. the top-most row) of $Q$, by the nature of the flips. This periodic pattern $Q'$ can be naturally extended to the whole plane by putting its translates side by side.
    As a corollary, the combinatorial entropy and the topological entropy are identical for symmetric nearest-neighbor subshifts, since every locally admissible pattern of $[1,n]^2$ for $n\ge 2$ is also a globally admissible one.
    
    On the other hand, we can also reduce the resulting pattern $Q$ of $[1,2n-1]^2$ to $[1,n+1]^2$ (note that $2n-1\ge n+1$ for $n\ge 2$). In other words, we have extended $P$ to $[1,n+1]^2$. Therefore, $C_n\le C_{n+1}$ for every $n\ge 2$.

    The same argument applies in higher dimensions by following the analogous process in Section \ref{sec:higher}.
\end{proof}

Now suppose that we take $4$ patterns $P_1,P_2,P_3,P_4$ with the same state (the patterns themselves are not necessarily different). We apply the same process as above, with $P_1,H(P_2),V(P_3),H(V(P_4))$ in place of $P,H(P),V(P),H(V(P))$ instead. The resulting pattern is still admissible due to the same argument as previously:
\begin{center}
    {
\ttfamily
\begin{tabular}{|c|c|c|}
\hline
$H(P_2)$ & \begin{tabular}{@{}c@{}} i\\h\\g\\f \end{tabular} & $H(V(P_4))$ \\
\hline
a b c d & e & d c b a \\
\hline
$P_1$ & \begin{tabular}{@{}c@{}} f\\g\\h\\i \end{tabular} & $V(P_3)$ \\
\hline
\end{tabular}\;.
}
\end{center}

As the total number of combinations of $P_1,P_2,P_3,P_4$ is $\left(C_n^{(s)}\right)^4$, and each combination yields a distinct admissible pattern on $[1,2n-1]^2$, we have
\[
    C_{2n-1} \ge \sum_{s\in S} \left(C_n^{(s)}\right)^4,
\]
where $S$ denotes the set of states.

For some convenience in a recurrence later, we use the inequality for $n+1$ instead of $n$ and obtain
\[
    C_{2n+1} \ge \sum_{s\in S} \left(C_{n+1}^{(s)}\right)^4 
    \ge \frac{1}{|S|}\left(\sum_{s\in S} \left(C_{n+1}^{(s)}\right)^2\right)^2
    \ge \frac{1}{|S|^3}\left(\sum_{s\in S} C_{n+1}^{(s)}\right)^4
    \ge \frac{(C_{n+1})^4}{|\Sigma|^{3(2n+1)}},
\]
where both the second inequality and the third inequality are due to the Cauchy--Schwarz inequality. Note that $|S|\le |\Sigma|^{2n+1}$ since a state on $[1,n+1]^2$ is determined by the values on $2n+1$ points of $[1,n+1]^2\setminus [1,n]^2$.

Dividing by $|\Sigma|^{6n+1}$ gives
\[
    \frac{C_{2n+1}}{|\Sigma|^{6n+1}} \ge \frac{(C_{n+1})^4}{|\Sigma|^{12n+4}}.
\]
Taking the $(4n^2)$-th root, we have
\begin{equation} \label{eq:increasing-plane}
    \left(\frac{C_{2n+1}}{|\Sigma|^{3\cdot(2n)+1}}\right)^{\frac{1}{(2n)^2}} \ge \left(\frac{C_{n+1}}{|\Sigma|^{3n+1}}\right)^\frac{1}{n^2}.
\end{equation}

Let
\[
    u_n = \left(\frac{C_{n+1}}{|\Sigma|^{3n+1}}\right)^\frac{1}{n^2}.
\]
For every $n$, it follows from \eqref{eq:increasing-plane} that the sequence $u_n, u_{2n}, u_{4n}, u_{8n}, \dots$ is increasing. Meanwhile, the sequence $u_1,u_2,u_3,\dots$ converges to $e^h$. Therefore, for every $n$,
\[
    e^h\ge \left(\frac{C_{n+1}}{|\Sigma|^{3n+1}}\right)^\frac{1}{n^2},
\]
which concludes Theorem \ref{thm:plane}.

\begin{remark} \label{rem:concat}
    Our ``merging'' approach is somewhat similar to the usual concatenation approach, which is usually used to prove the upper bound in Theorem \ref{thm:plane}. In particular, concatenating the four previously mentioned patterns $P_1,P_2,P_3,P_4$, we obtain
\begin{center}
    {
\ttfamily\upshape
\begin{tabular}{|c|c|c|c|c|}
\hline
a b c d & e & a b c d & e \\
\hline
$P_2$ & \begin{tabular}{@{}c@{}} f\\g\\h\\i \end{tabular} & $P_4$ & \begin{tabular}{@{}c@{}} f\\g\\h\\i \end{tabular} \\
\hline
a b c d & e & a b c d & e \\
\hline
$P_1$ & \begin{tabular}{@{}c@{}} f\\g\\h\\i \end{tabular} & $P_3$ & \begin{tabular}{@{}c@{}} f\\g\\h\\i \end{tabular} \\
\hline
\end{tabular}\;.
}
\end{center}
For comparison, instead of just putting the patterns side by side for an upper bound, we flip the patterns and merge the boundaries appropriately to ensure the admissibility of the resulting pattern for a lower bound. On the other hand, we have not covered the case of merging patterns of different states.
\end{remark}

\section{In higher dimensions}
\label{sec:higher}
In this section, we prove Theorem \ref{thm:main}. The adaptation of the techniques in Section \ref{sec:plane} for higher dimensions is straightforward with $d$ flips for $\mathbb Z^d$, instead of only the vertical flip and the horizontal flip as in $\mathbb Z^2$. We only need to formalize the notation and the steps as follows. However, the analysis may be slightly more involved.

We denote by $f_1,\dots,f_d$ the $d$ flips corresponding to the $d$ dimensions. In particular, for a cubic pattern $P\in\Sigma^{[1,n]^d}$, the function $f_k$ swaps the parts of the pattern on the hyperplanes with the $k$-th coordinates equal to $i$ and $n+1-i$. In other words, if $Q=f_k(P)$ then for every point $x=(x_1,\dots,x_d)\in [1,n]^d$, we have $Q(x) = P(y)$ for
\[
    y=(x_1,\dots,x_{k-1},n+1-x_k,x_{k+1},\dots,x_d).
\]
The notation $P(x)$ here denotes the value of the point $x$ in the pattern $P$.

By a state of a pattern $P\in\Sigma^{[1,n]^d}$ we mean the values of the points $x=(x_1,\dots,x_d)$ where $x_i=n$ for at least one $i$. In other words, those points are in $[1,n]^d\setminus [1,n-1]^d$. Let $S$ denote the set of all possible states.

Now take any $2^d$ patterns $P_0,\dots,P_{2^d-1}\in \Sigma^{[1,n]^d}$ with the same state. For each number $t=0,1,\dots,2^d-1$, we consider the binary representation, which is a unique sequence of binary digits $b_1,\dots,b_d$. In other words, each $t$ determines a unique subset of $\{f_1,\dots,f_d\}$, where $f_k$ belongs to the subset if $b_k=1$. Let $\tau_t$ denote the composition of all the flips in the subset for $t$. Note that the order of flips in the composition does not matter. Also, when $t=0$, there is no flip and $\tau_t$ is the identity.

For each $t$, we consider the result of the transform $\tau_t(P_t)$ and translate it by the vector $\sum_i (n-1) b_ie_i$, where each $e_i$ is the unit vector with the $i$-th coordinate being $1$. As the subshift is symmetric, each $\tau_t(P_t)$ is also an admissible pattern, since each flip keeps the admissibility of the pattern. The union of all these
\[
    \bigcup_{t=0}^{2^d-1} \left(\tau_t(P_t) + \sum_{i=1}^d (n-1)b_ie_i\right)
\]
is a pattern for $[1,2n-1]^d$. Its $2^d$ components are disjoint except at the boundary where they agree on the values since all the $P_i$ share the same state. 
Every nearest-neighbor pair of sites in \([1,2n-1]^d\) is contained in at
least one of the components. Since each component is
admissible, the union is also admissible.

Let $C_n$ (resp. $C_n^{(s)}$ for a state $s\in S$) denote the number of patterns in $\Sigma^{[1,n]^d}$ (resp. with the state $s$). Since the above manipulation of $2^d$ patterns of $[1,n]^d$ yields a unique pattern of $[1,2n-1]^d$, we have
\begin{equation} \label{eq:key}
    C_{2n-1}\ge \sum_{s\in S} \left(C_n^{(s)}\right)^{2^d}.
\end{equation}
We do not have the equality in general, since a pattern in $[1,2n-1]^d$ may be the composition of patterns of different states. 

We can eliminate the sum with the states in \eqref{eq:key} as follows. For some convenience later, we write it for $n+1$ instead of $n$.
\begin{lemma} \label{lem:key}
    For every $n$,
    \[
        C_{2n+1} \ge \frac{(C_{n+1})^{2^d}}{|\Sigma|^{(2^d-1)((n+1)^d-n^d)}}.
    \]
\end{lemma}

\begin{proof}
We apply Inequality \eqref{eq:key} for $n+1$ and obtain
\[
    C_{2n+1}
    \ge
    \sum_{s\in S} \left(C_{n+1}^{(s)}\right)^{2^d}.
\]
By the power mean inequality,\footnote{We can apply the Cauchy-Schwarz inequality multiple times as in Section \ref{sec:plane} to obtain the same result, but using the power mean inequality is more concise here.} for any nonnegative numbers $(a_s)_{s\in S}$ and
any $p\ge 1$,
\[
    \sum_{s\in S} a_s^p
    \ge
    |S|^{1-p}\left(\sum_{s\in S}a_s\right)^p.
\]
Applying this with $a_s=C_{n+1}^{(s)}$ and $p=2^d$, we get
\[
    \sum_{s\in S} \left(C_{n+1}^{(s)}\right)^{2^d}
    \ge
    |S|^{1-2^d}
    \left(\sum_{s\in S} C_{n+1}^{(s)}\right)^{2^d}
    =
    |S|^{1-2^d}(C_{n+1})^{2^d}.
\]
Hence
\[
    C_{2n+1}
    \ge
    \frac{(C_{n+1})^{2^d}}{|S|^{2^d-1}} \ge \frac{(C_{n+1})^{2^d}}
    {|\Sigma|^{(2^d-1)((n+1)^d-n^d)}},
\]
where the bound $|S|\le |\Sigma|^{(n+1)^d-n^d}$ in the second step is due to the fact that a state on $[1,n+1]^d$ is determined by the values on $[1,n+1]^d\setminus [1,n]^d$.
\end{proof}

Denote
\[
    q_d(n)=(2^d-1)\sum_{k=0}^{d-1} \frac{\binom{d}{k}}{2^d-2^k}\,n^k.
\]
The following relation can be verified in a simple but technical way, as in Appendix \ref{app:verification}.

\begin{lemma} \label{lem:relation}
For every $d,n$, we have
\[
    q_d(2n) + (2^d-1)((n+1)^d-n^d) = 2^d q_d(n).
\]
\end{lemma}

It follows from Lemma \ref{lem:key} and Lemma \ref{lem:relation} that
\[
    \frac{C_{2n+1}}{|\Sigma|^{q_d(2n)}} \ge \frac{(C_{n+1})^{2^d}}{|\Sigma|^{q_d(2n) +(2^d-1)((n+1)^d-n^d)}} = \frac{(C_{n+1})^{2^d}}{|\Sigma|^{2^d q_d(n)}}.
\]
Taking the root of degree $(2n)^d$, one obtains
\begin{equation} \label{eq:increasing}
    \left(\frac{C_{2n+1}}{|\Sigma|^{q_d(2n)}}\right)^{\frac{1}{(2n)^d}} \ge \left(\frac{C_{n+1}}{|\Sigma|^{q_d(n)}}\right)^{\frac{1}{n^d}}.
\end{equation}

Let
\[
    v_n = \left(\frac{C_{n+1}}{|\Sigma|^{q_d(n)}}\right)^{\frac{1}{n^d}}.
\]
For every $n$, it follows from \eqref{eq:increasing} that the sequence $v_n, v_{2n}, v_{4n}, v_{8n}, \dots$ is increasing. Meanwhile, the sequence $v_1,v_2,v_3,\dots$ converges to $e^h$. Therefore, for every $n$,
\[
    e^h\ge \left(\frac{C_{n+1}}{|\Sigma|^{q_d(n)}}\right)^{\frac{1}{n^d}},
\]
which concludes Theorem \ref{thm:main}.

\appendix
\section{Verification of Lemma \ref{lem:relation}}
\label{app:verification}
Starting from
\[
q_d(n)=(2^d-1)\sum_{k=0}^{d-1}\frac{\binom{d}{k}}{2^d-2^k} n^k,
\]
we compute
\[
q_d(2n)
=(2^d-1)\sum_{k=0}^{d-1}\frac{\binom{d}{k}}{2^d-2^k}(2n)^k
=(2^d-1)\sum_{k=0}^{d-1}\frac{2^k\binom{d}{k}}{2^d-2^k}n^k.
\]

Also,
\[
(n+1)^d-n^d
=\sum_{k=0}^{d-1}\binom{d}{k}n^k.
\]
Hence
\[
q_d(2n)+(2^d-1)\bigl((n+1)^d-n^d\bigr)
\]
equals
\[
(2^d-1)\sum_{k=0}^{d-1}
\left(
\frac{2^k\binom{d}{k}}{2^d-2^k}
+\binom{d}{k}
\right)n^k.
\]
Combining the terms inside the parentheses,
\[
\frac{2^k}{2^d-2^k}+1
=
\frac{2^k+2^d-2^k}{2^d-2^k}
=
\frac{2^d}{2^d-2^k}.
\]
Therefore
\[
q_d(2n)+(2^d-1)\bigl((n+1)^d-n^d\bigr)
=
(2^d-1)\sum_{k=0}^{d-1}
\frac{2^d\binom{d}{k}}{2^d-2^k}n^k,
\]
which is exactly
\[
2^d q_d(n).
\]

\bibliographystyle{unsrt}
\bibliography{symsft}

\end{document}